\documentclass{birkjour}

\usepackage{amsmath}
\usepackage{amsthm}
\usepackage{amsfonts}
\usepackage{graphicx}

\newtheorem{lemma}{Lemma}
\newtheorem{prop}{Proposition}
\newtheorem{theorem}{Theorem}
\newtheorem{cor}{Corollary}

\newcommand{\bu}{\mathbf{u}}
\newcommand{\bn}{\mathbf{n}}
\newcommand{\ds}{\mathbf{ds}}
\newcommand{\np}{\nabla^{\bot}}
\newcommand{\Lv}{\operatorname{L}_{\textrm{vor}}}

\newcommand{\Lo}{\operatorname{L}_0}
\newcommand{\Lw}{\widehat{\operatorname{L}}_0}
\newcommand{\eL}{\operatorname{L}}
\newcommand{\K}{\operatorname{K}}
\DeclareMathOperator{\curl}{curl}
\DeclareMathOperator{\dv}{div}

\begin{document}

\title{The $L^2$ essential spectrum of the 2D Euler operator}
\author{Graham Cox}
\email{ghcox@email.unc.edu}
\address{Department of Mathematics, UNC Chapel Hill, Phillips Hall CB \#3250, Chapel Hill, NC 27599}

\begin{abstract} Even in two dimensions, the spectrum of the linearized Euler operator is notoriously hard to compute. In this paper we give a new geometric calculation of the essential spectrum for 2D flows. This generalizes existing results---which are only available when the flow has arbitrarily long periodic orbits---and clarifies the role of individual streamlines in generating essential spectra.
\end{abstract}

\thanks{The author would like to thank Yuri Latushkin for many helpful discussions during the preparation of this work, and the referees for their comments and suggestions during the review process. This research has been supported by the Office of Naval Research under the MURI grant N00014-11-1-0087.}

\maketitle

\section{Introduction}
The Euler equation gives the time evolution of an incompressible fluid of constant density as
\begin{align}
	\frac{\partial \bu}{\partial t} + (\bu \cdot \nabla) \bu &= - \nabla p \label{eqn:Euler} \\
	\nabla \cdot \bu &= 0 \nonumber
\end{align}
where $\bu$ is the fluid velocity and $p$ is the pressure. For a simply-connected domain $\Omega \subset \mathbb{R}^2$ with tangential boundary conditions, $\bu \cdot \bn = 0$, the velocity can be written as $\bu = - \np \psi$, where $\psi$ is the stream function. Then the vorticity, $\omega = \curl \bu = - \Delta \psi$, evolves according to
\begin{align}
	\frac{\partial \omega}{\partial t} = \np \psi \cdot \nabla \omega. \label{eqn:vor}
\end{align}
For a non-simply connected domain, the relationship between the velocity and vorticity formulations is more subtle, and depends on the choice of boundary conditions. While not necessary for the results presented here (which deal only with the vorticity operator), this issue is of independent interest, and we refer the reader to \cite{LN10} for further details.

In vorticity form, the linearized Euler operator at a steady state, $\omega_0 = - \Delta \psi_0$, can be written as
\begin{align*}
	\Lv = \Lo + \K
\end{align*}
where $\Lo = \np \psi_0 \cdot \nabla$ and $\K = - \nabla \omega_0 \cdot \curl^{-1}$. The operator $\curl^{-1}$ is defined to be $\nabla^{\bot} \circ \Delta^{-1}$, where $\Delta^{-1}$ denotes the inverse  Laplacian on the space of zero mean functions, subject to the boundary conditions
\begin{align}
	\left. \psi \right|_{\Sigma_i} = \text{ const. and } \int_{\Sigma_i} \frac{\partial \psi}{\partial n} = 0,
	\label{BC}
\end{align}
where $\{\Sigma_i\}$ are the connected components of $\partial \Omega$.
These boundary conditions (which are elaborated on in Section 5 of \cite{L04}) ensure that the corresponding velocity field, $\bu = \curl^{-1} \omega$, satisfies $\bu \cdot \bn = 0$ and $\int_{\Sigma_i} \bu \cdot \ds = 0$ on each component of the boundary.

Since $\Lv$ is degenerate and non-elliptic, it is quite difficult to find conditions that guarantee the existence of an unstable eigenvalue. The earliest results in this area \cite{F50,H64,R80} give necessary conditions for instability in the special case of parallel shear flows. More recent results \cite{BM99,BFY99,FVY00,L03} give sufficient conditions for instability in a number of special cases. A more general sufficient condition appears in \cite{L04}, but this appears to be difficult to evaluate in many cases of interest. For further details, an excellent recent overview of stability results can be found in \cite{SF05}.

Recently, progress has been made in computing the essential spectrum, which is related to shortwave perturbations of the velocity field. In \cite{LS03,LS05} it was shown that the $H^m$ essential spectrum comprises a vertical strip (for $m \neq 0$), with the width determined by the fluid Lyapnuov exponent. When computed in $L^2$ the essential spectrum must be purely imaginary, because $\Lo$ is skew-adjoint. In fact it is known that $\sigma_{\textrm{ess}}(\Lv) = i \mathbb{R}$ if the underlying flow has arbitrarily long trajectories. The proof of this result, in \cite{LS03}, uses an approximate eigenfunction construction supported in a neighborhood of a trajectory connecting two hyperbolic fixed points of the underlying fluid flow.

The primary contribution of the current paper is a complete description of the $L^2$ essential spectrum in the absence of arbitrarily long trajectories, when such approximate eigenfunction constructions are no longer available. Our method also applies to flows with arbitrarily long trajectories, as long as ``most" orbits are either constant or periodic. This gives a new perspective on the results of \cite{LS03}, demonstrating explicitly how each periodic orbit contributes to the essential spectrum. The proof is geometric in nature, using the co-area formula (which was applied in a different context in \cite{L04}) to decompose the spatial domain into level sets of the stream function, and hence reduce the spectral problem to a family of one-dimensional, self-adjoint problems.

For the remainder of the paper we take $\Omega$ to be any one of the following:
\begin{enumerate}
	\item a smooth, bounded domain $\Omega \subset \mathbb{R}^2$ (not necessarily simply-connected);
	\item a bounded cylinder, i.e. $[0,L] \times [a,b]$ with the identification $(0,y) \sim (L,y)$ for $y \in [a,b]$;
	\item the torus $\mathbb{T}^2$.
\end{enumerate}
When $\Omega$ has a nontrivial boundary (i.e. in cases (1) and (2) above) we assume that $\nabla^{\bot} \psi_0 \cdot \bn = 0$; this ensures that $\psi_0$ is constant on each connected component of the boundary.

Our main result concerns the essential spectrum of $\Lv$ acting on the space of zero mean functions $X_0 \subset L^2(\Omega)$. (A precise definition of the unbounded operator $\Lv$ will be given in Section \ref{sec:def}.)
\begin{theorem} Let $\Omega$ be one of the domains listed above, with $\psi_0 \in C^3(\bar{\Omega})$ a nonconstant steady solution to the Euler equation. If the flow generated by $\nabla \psi_0$ does not have arbitrarily long trajectories, then
\begin{align}
	\sigma_{\textrm{ess}}(\Lv; X_0) = i \overline{\left\{ \frac{2\pi k}{T} : k \in \mathbb{Z}, \text{ the flow has an orbit with period } T \right\}}.
	\label{spec}
\end{align}
If the flow has arbitrarily long trajectories, then $\sigma_{\textrm{ess}}(\Lv; X_0) = i \mathbb{R}$.
\label{thm:main}
\end{theorem}

When the set of periods comprises an interval $(T_{\textrm{min}}, T_{\textrm{max}})$, the expression in (\ref{spec}) reduces to
\begin{align*}
	i \sigma_{\textrm{ess}}(\Lv; X_0) = \left( \bigcup_{-k \in \mathbb{N}} \left[ \frac{2 \pi k}{T_{\textrm{min}}}, \frac{2 \pi k}{T_{\textrm{max}}} \right]  \right) \cup \{0\} \cup \left( \bigcup_{k \in \mathbb{N}} \left[ \frac{2 \pi k}{T_{\textrm{max}}}, \frac{2 \pi k}{T_{\textrm{min}}} \right] \right),
\end{align*}
and so when $T_{\textrm{min}} = 0$ we find that
\begin{align*}
	i\sigma_{\textrm{ess}}(\Lv) = \left(-\infty, -\frac{2 \pi}{T_{\textrm{max}}} \right] \cup \{0 \} \cup \left[ \frac{2 \pi}{T_{\textrm{max}}}, \infty \right).
\end{align*}
In general, there are at least two spectral gaps along the imaginary axis when the flow does not have arbitrarily long trajectories, and it is not difficult to compute the exact number of gaps in terms of the maximum and minimum periods. In particular, the number of gaps is always finite, except in the degenerate case $T_{\textrm{min}} = T_{\textrm{max}}$. The spectral inclusion theorem (see \cite{EN00}) thus yields the following result for the corresponding evolution semigroup.

\begin{cor} If $T_{\textrm{min}} \neq T_{\textrm{max}}$, then the spectrum of $e^{t \Lv}$ contains the unit circle.
\end{cor}

Therefore the spectrum of the semigroup will contain the entire unit circle (except in the degenerate case described above), whether the flow has arbitrarily long trajectories or not.


\section{Definitions and notation}
\label{sec:def}
We let $X = L^2(\Omega)$ and
\begin{align}
	D := \{\omega \in L^2(\Omega) : \Lo \omega \in L^2(\Omega) \},
	\label{dom}
\end{align}
with $X_0$ and $D_0$ the respective subspaces of zero mean functions. The inclusion $\Lo \omega \in L^2(\Omega)$ in (\ref{dom}) is meant distributionally, i.e. there exists $f \in L^2(\Omega)$ such that
\begin{align}
	\int_{\Omega} f \phi = - \int_{\Omega} \omega \left( \Lo \phi \right)
	\label{eqn:Lweak}
\end{align}
for all $\phi \in C^{\infty}(\bar{\Omega})$. If $\omega \in C^1(\bar{\Omega})$, then $f = \Lo \omega$ classically. To see this, it suffices to note that $\text{div}(\omega \phi \nabla^{\bot} \psi_0) = \omega \Lo \phi + \phi \Lo \omega$, and so the boundary term that comes from Stokes' theorem is proportional to $\nabla^{\bot} \psi_0 \cdot \bn$, which vanishes by assumption. It will be shown below that $\omega \in X_0$ implies $\K \omega \in X_0$, hence $\Lv \omega \in L^2(\Omega)$ if and only if $\Lo \omega \in L^2(\Omega)$. We are thus justified in defining $\Lv$ as an unbounded operator on $X_0$ with domain $D_0$.

We define the essential spectrum, $\sigma_{\textrm{ess}}(\Lv; X_0)$, to be the set of all $\lambda \in \mathbb{C}$ such that $\Lv - \lambda I$ is not Fredholm. (For a discussion of alternate definitions of the essential spectrum, see \cite{EE87} and \cite{K95}.)

For $\psi_0 \in C^3(\bar{\Omega})$ a steady state of the 2D Euler equation, let $\widehat{\Omega}$ denote the union of (the images of) all nontrivial periodic orbits for the flow generated by $\nabla^{\bot} \psi_0$, and $\Omega_0 = \{ \nabla \psi_0 = 0 \}$ the set of all fixed points.

\begin{lemma} With $\psi_0$ as above, $\widehat{\Omega}$ is an open subset of $\Omega$.
\end{lemma}

\begin{proof}
Let $x_0 \in \widehat{\Omega}$ be a point on a nontrivial periodic orbit $\Gamma$, which by assumption is diffeomorphic to $\mathbb{S}^1$, and contains no fixed points. Thus by continuity $\nabla \psi_0$ is nonzero on an open set $S \subset \Omega$ that contains $\Gamma$. Let $a = \psi_0(x_0)$.  We can assume (by shrinking $S$ if necessary) that $\psi_0|_S^{-1}(a) = \Gamma$, i.e. $\Gamma$ is the only connected component of the level set $\psi_0^{-1}(a)$ that is contained in $S$.

A standard argument using the normalized gradient flow of $\psi_0$ (see \cite{M63} for details) then shows that, for some $\epsilon > 0$, $\psi_0|_S^{-1} (a - \epsilon, a + \epsilon)$ forms a tubular neighborhood around $\Gamma$, with $\psi_0|_S^{-1}(a+t)$ diffeomorphic to $\Gamma$ for $|t| < \epsilon$. We have thus constructed an open set in $\Omega$, consisting of the images of nontrivial periodic orbits, that contains $\Gamma$ (and hence $x$), so the proof is complete.
\end{proof}

In proving Theorem \ref{thm:main} it will sometimes be convenient to impose the following assumption on the geometry of the flow:
\begin{align*}
 	\textrm{the set  } \Omega \setminus \left( \widehat{\Omega} \cup \Omega_0 \right) \textrm{ has measure zero in } \Omega. \tag{H}
\label{hyp:period}
\end{align*}
Since every trajectory must be diffeomorphic to a point, a circle or the real line, we are thus assuming that the image of the set of trajectories diffeomorphic to the line has measure zero. 

We then define $\widehat{X} := L^2(\widehat{\Omega})$, and let $\Lw$ denote the restriction of $\Lo$ to $\widehat{X}$, with domain
\begin{align}
	\widehat{D} =  \{ \omega \in L^2(\widehat{\Omega}) : \Lw \omega \in L^2(\widehat{\Omega}) \},
\end{align}
where $\Lw$ is defined distributionally, as in (\ref{eqn:Lweak}), using smooth test functions on the closure of $\widehat{\Omega}$. This agrees with the classical definition of $\Lw$ for functions of class $C^1$ on the closure of $\widehat{\Omega}$, though the argument is more subtle than before (\textit{cf.} the discussion following (\ref{eqn:Lweak})) because the boundary of $\widehat{\Omega}$ is not necessarily smooth. To prove this, it suffices to examine a component of $\widehat{\Omega}$, say $\psi_0^{-1}(a,b)$, and observe that
\begin{align*}
	\int_{\psi_0^{-1}(a+\epsilon, b-\epsilon) } (\Lo \omega) \phi = -\int_{\psi_0^{-1}(a+\epsilon, b-\epsilon) } \omega \Lo \phi
\end{align*}
for any $\epsilon > 0$ by Stokes' theorem, because the boundary term is proportional to $\nabla^{\bot} \psi_0 \cdot \bn$ and hence vanishes on any level set of $\psi_0$. This equality holds in the $\epsilon = 0$ limit because $\omega$ and $\phi$ (and their first derivatives) are bounded on the closure of $\widehat{\Omega}$.

We define the index set $\mathcal{J}$ to be the disjoint union
\begin{align}
	\mathcal{J} := \coprod_{i} \psi_0 \left( \widehat{\Omega}_i \right)
	\label{eqn:index}
\end{align}
where $\{ \widehat{\Omega}_i \}$ are the connected components of $\widehat{\Omega}$. This set records all values assumed by the stream function $\psi_0$ on periodic orbits, with multiplicity. The stream function is nonsingular on $\widehat{\Omega}$, hence the index set is a disjoint union of open intervals. We can thus define a measure and a topology on $\mathcal{J}$ that agree with the Lebesgue measure and Euclidean topology on each component of the disjoint union.

For a simple example, suppose $\Omega = B_{2\pi}(0) \subset \mathbb{R}^2$ and $\psi_0 = \cos r$ in polar coordinates. Then $\Omega_0 = \{0\} \cup \{r = \pi \}$ and $\widehat{\Omega} = \{0 < r < \pi \} \cup \{\pi < r < 2\pi \} = \widehat{\Omega}_1 \cup \widehat{\Omega}_2$, so we find that
\begin{align*}
	\mathcal{J} =  (-1,1) \coprod (-1,1).
\end{align*}

It is important to note that this is a disjoint union, so the stream function $\psi_0$ gives a bijection between $\mathcal{J}$ and the set of periodic orbits. The period function, $T: \mathcal{J} \rightarrow (0,\infty)$, is then defined by letting $T(\rho)$ be the period of the orbit corresponding to $\rho \in \mathcal{J}$. The period function is not necessarily injective, so the preimage under $T$ of a set in $(0,\infty)$ could contain elements in different components of the disjoint union defining $\mathcal{J}$.


\section{Analytic preliminaries}
In this section we relate the essential spectrum of $\Lv$ to the spectrum of $\Lw$. This is done because $\Lw$ can be easily understood using the geometric techniques developed in Section \ref{sec:geom}.

\begin{prop} The spectrum satisfies $\sigma(\Lw; \widehat{X}) \subseteq \sigma_{\textrm{ess}}(\Lv; X_0)$, with equality if (\ref{hyp:period}) is satisfied.
\end{prop}

This proposition will be proved by a number of straightforward lemmas which together establish the string of equalities
\begin{align*}
	\sigma(\Lw; \widehat{X})  =\sigma(\Lo; X) = \sigma(\Lo; X_0) = \sigma_{\textrm{ess}}(\Lo; X_0) = \sigma_{\textrm{ess}}(\Lv; X_0).
\end{align*}
We first observe that $T := i \Lo$ is self-adjoint. 

\begin{lemma} The operator $T := i \Lo$, with domain $D(T) := D_0$, is self-adjoint.
\end{lemma}

\begin{proof}
Suppose $u \in D(T^*)$. Then by definition there exists $u^* \in X_0$ so that $\left<u, T \omega \right> = \left<u^*, \omega \right>$ for all $\omega \in D(T)$, hence for all $\omega \in C^{\infty}(\bar{\Omega}) \cap X_0$. Since $T$ vanishes on constant functions and $\left<u^*,1\right> = 0$ for $u^* \in X_0$, we have by linearity that $\left<u, T \omega \right> = \left<u^*, \omega \right>$ for all $\omega \in C^{\infty}(\bar{\Omega})$. This means
\begin{align*}
	\int_{\Omega} u^* \omega  = -i \int_{\Omega} u \left( \Lo \omega \right)
\end{align*}
and so by definition $u \in D(T)$ and $T^*u := u^* = i \Lo u$. We thus have $D(T) = D(T^*)$ and $T = T^*$ as claimed.
\end{proof}

In particular, this implies that $\Lo$ is a closed operator, so we can use a suitable form of Weyl's theorem to simplify our computation of the essential spectrum.

\begin{lemma} The essential spectrum satisfies $\sigma_{\textrm{ess}}(\Lv; X_0) = \sigma_{\textrm{ess}}(\Lo; X_0)$.
\label{lem:compact}
\end{lemma}

\begin{proof} 
By definition we have $\Lv = \Lo + K$, so the result will follow from Weyl's theorem (Theorem 5.35 of \cite{K95}) once we establish the compactness of $K: X_0 \rightarrow X_0$.

For each $\omega \in X_0$, there is a unique zero mean solution $\psi$ to the equation $-\Delta \psi = \omega$, subject to the boundary conditions (\ref{BC}). Moreover, we have $\psi \in H^2(\Omega)$, with the uniform estimate
\begin{align*}
	\| \psi \|_{H^2(\Omega)} \leq C \| \omega \|_{L^2(\Omega)}.
\end{align*}
(This is an immediate consequence of the results in Sections 7.E--7.G of \cite{F95}, applied to the Dirichlet form
\begin{align*}
	D(u,v) = \int_{\Omega} \nabla u \cdot \nabla v
\end{align*}
on the space $\mathcal{X} = \{\psi \in H^1(\Omega) : \psi|_{\Sigma_i} = \text{ const. for each } i \}$.)

We thus have that $\curl^{-1}: X_0 \rightarrow H^1(\Omega)$ is bounded. Given a bounded sequence $\{\omega_i\}$ in $X_0$, the Rellich--Kondrachov theorem implies $\{\curl^{-1} \omega_i\}$ has a convergent subsequence in $L^2(\Omega)$. Recalling that $K\omega = -\nabla \omega_0 \cdot \curl^{-1} \omega$, with $\nabla \omega_0$ bounded (because $\psi_0 \in C^3$), we conclude that $\{K\omega_i\}$ has an $L^2$-convergent subsequence, hence $K: X_0 \rightarrow X$ is compact.

To complete the proof we must show that the range of $K$ is contained in $X_0$. Observing that $K\omega = -\dv[ \omega_0 \nabla^{\bot} (\Delta^{-1} \omega)]$, we have from Stokes' theorem that $\int_{\Omega} K\omega = 0$ because $\Delta^{-1} \omega$ is constant on each component of $\partial \Omega$.
\end{proof}

We next use the fact that $\Lo$ has an infinite-dimensional kernel to show that it has no discrete spectrum.

\begin{lemma} The spectrum of $\Lo$ satisfies $\sigma_{\textrm{ess}}(\Lo; X_0) = \sigma(\Lo; X_0)$.
\end{lemma}

\begin{proof} It suffices to show that $\sigma(\Lo; X_0)$ contains no eigenvalues of finite multiplicity. Suppose such an eigenvalue $\lambda$ exists, with eigenfunction $\omega \in D_0$. We first assume that $\lambda \neq 0$. For any positive integer $n$ it follows that $\Lo ( \psi_0^n \omega ) = \lambda \psi_0^n \omega$ (because $\Lo \psi_0  = 0$), and Stokes' theorem implies that $\psi_0^n \omega$ has zero mean. We will contradict the finite multiplicity of $\lambda$ by proving that the set $\{\psi_0^n \omega \}$ is linearly independent

If this were not the case, there would exist nonzero constants $\{a_i\}$ and positive integers $\{n_i\}$ so that
\begin{align*}
	\left( \sum_{i=1}^N a_i \psi_0^{n_i} \right) \omega= 0
\end{align*}
almost everywhere, hence
\begin{align*}
	\sum_{i=1}^N a_i \psi_0^{n_i} = 0
\end{align*}
on the set $\{ \omega \neq 0\}$. To complete the proof, we show that $\psi_0$ assumes uncountably many values on the set $\{ \omega \neq 0\}$, and so there are an infinite number of roots for the polynomial equation $\sum a_i x^{n_i} = 0$. This implies $a_1 = \cdots = a_N = 0$, thus establishing the linear independence of the set $\{ \psi_0^n \omega \}$.

We again proceed by contradiction, assuming the set $\psi_0 \left[ \{ \omega \neq 0\} \right] \subset \mathbb{R}$ has zero measure, and hence $\{ \omega \neq 0 \} \cap \{ \nabla \psi_0 \neq 0 \} \subset \Omega$ also has zero measure. Then for any test function $\phi \in C^{\infty}(\bar{\Omega})$ we have
\begin{align*}
	\lambda \int_{\Omega} \omega \phi &= - \int_{\Omega} \omega \Lo \phi \\
		&= 0
\end{align*}
because $\Lo \phi = 0$ where $\nabla \psi_0 = 0$ and $\omega = 0$ where $\nabla \psi_0 \neq 0$. Recalling that $\lambda \neq 0$, this implies $\omega = 0$, which gives the desired contradiction.

When $\lambda =0$, we can consider the set of eigenfunctions
\begin{align*}
	\psi_0^n - \frac{\int_{\Omega} \psi_0^n}{\int_{\Omega} 1}
\end{align*}
which are in $D_0$ by construction. These are easily seen to be linearly independent (as above) because $\psi_0$ is nonconstant.
\end{proof}

We next show that, for spectral purposes, it suffices to consider $\Lo: D \rightarrow X$, without the zero mean restriction.

\begin{lemma} The spectrum of $\Lo$ satisfies $\sigma(\Lo; X_0) = \sigma(\Lo; X)$.
\end{lemma}

\begin{proof} For the duration of the proof we let $\eL$ denote the operator $\Lo: D \rightarrow X$. We proceed by showing that the resolvent sets of $\Lo$ and $\eL$ coincide. Assuming without loss of generality that $\int_{\Omega} \psi_0 = 0$, the stream function satisfies $\Lo \psi_0 = 0$, with $\psi_0 \in D_0 \cap D$, hence $0 \in \sigma(\Lo) \cap \sigma(\eL)$.

We let $P: X \rightarrow X_0$ denote $L^2$-orthogonal projection, and $J: X_0 \rightarrow X$ inclusion. It is easily shown that $\Lo = PLJ$ and $\eL = J \Lo P$.
Then if $\lambda \in \rho(\eL)$ we have
\begin{align*}
	(\Lo - \lambda I)^{-1} = P (\eL - \lambda I)^{-1} J
\end{align*}
and if $\lambda \in \rho(\Lo)$ we have
\begin{align*}
	(\eL - \lambda I)^{-1} = J(\Lo - \lambda)^{-1} P - \lambda^{-1} (I-JP)
\end{align*}
so the resolvent sets agree.
\end{proof}

For the purpose of determining the $L^2$ essential spectrum, it is possible to discard those parts of the domain on which the steady-state stream function is singular---this works because the coefficients of $\Lo$ vanish at the singular points of $\psi_0$. 

We let $E$ denote extension by zero from $\widehat{X}$ to $X$, with $R$ the corresponding restriction from $X$ to $\widehat{X}$. It is clear that $E$ and $R$ are bounded, with $E^* = R$. Because $\widehat{\Omega} \subset \Omega$ is open and $\Lo$ is trivial outside $\widehat{\Omega}$ (up to a set of measure zero), $\Lo$ is well-behaved with respect to extension and restriction, in the sense made precise by the following lemma.

\begin{lemma} For any steady state $\psi_0$ we have $E(\widehat{D}) \subset D$ and $R(D) \subset \widehat{D}$, and
\begin{align}
	\Lo Ef &= E \Lw f  \label{eqn:EL} \\
	R \Lo g &= \Lw R g \label{eqn:RL}
\end{align}
for all $f \in \widehat{D}$ and $g \in D$, respectively. If (\ref{hyp:period}) is satisfied, then
\begin{align}
	\Lo g = \Lo ER g
	\label{eqn:Low}
\end{align}
for all $g \in D$.
\label{lemma:proj}
\end{lemma}

Using the fact that $RE$ is the identity on $\widehat{X}$, it follows immediately that
\begin{align}
	\Lw f = R \Lo E f
	\label{LWequiv}
\end{align}
for $f \in \widehat{D}$. From (\ref{eqn:EL}) and (\ref{eqn:RL}) we have $\Lo ER = ER \Lo = E \Lw R$ and so, when (\ref{hyp:period}) is satisfied,
\begin{align}
	\Lo g = E \Lw R g
	\label{L0equiv}
\end{align}
for all $g \in D$.

\begin{proof} It is clear that (\ref{eqn:RL}) holds for all differentiable functions, so for any test function $\phi \in C^{\infty}(\bar{\Omega})$ and $\omega \in \widehat{D}$ we have
\begin{align*}
	\int_{\Omega} \phi (E \Lw \omega) &= \int_{\widehat{\Omega}} (R \phi) (\Lw \omega) \\
	&= - \int_{\widehat{\Omega}} \omega \Lw(R\phi) \\
	&= - \int_{\widehat{\Omega}} \omega R(\Lw \phi) \\
	&= - \int_{\Omega} (E \omega) ( \Lw \phi).
\end{align*}
Thus, from the distributional definition of $\Lo$, we have that $E \omega \in D$, and $\Lo E \omega = E \Lw \omega$. This establishes (\ref{eqn:EL}); the proof of (\ref{eqn:RL}) is similar, starting with the fact that (\ref{eqn:EL}) holds in $\widehat{D}$, and hence for all test functions on $\widehat{\Omega}$.

To prove (\ref{eqn:Low}), we first observe that each test function $\phi$ on $\Omega$ satisfies $\Lo \phi  =0$ a.e. in $\widehat{\Omega}^c$. This implies $ER(\omega \Lo \phi) = \omega \Lo \phi$ a.e. in $\Omega$ for any $\omega \in X$, hence
\begin{align*}
	\int_{\Omega} (ER \omega)(\Lo \phi) &= \int_{\widehat{\Omega}} (R \omega)(R \Lo \phi) \\
		&= \int_{\widehat{\Omega}} R (\omega \Lo \phi) \\
		&= \int_{\Omega} ER (\omega \Lo \phi) \\
		&= \int_{\Omega} \omega \Lo \phi.
\end{align*}
For $\omega \in D$ this implies $\Lo ER \omega = \Lo \omega$, as claimed.
\end{proof}

The final result we need before comparing the spectra of $\Lo$ and $\Lw$ is the following localization property for $(\Lo - \lambda)^{-1}$. 

\begin{lemma} If $\lambda \in \rho(\Lo; X)$, then
\begin{align}
	ER (\Lo - \lambda I)^{-1} Ef = (\Lo - \lambda)^{-1} Ef
	\label{invlocal}
\end{align}
for all $f \in \widehat{X}$.
\end{lemma}
Thus if $u \in D$ solves the equation $(\Lo - \lambda I) u = Ef$ for some $f \in \widehat{X}$, and $\lambda$ is in the resolvent set of $\Lo$, then $u = 0$ a.e. in $\widehat{\Omega}^c$. We note that the relation
\begin{align}
	ER(\Lo - \lambda I)Ef = (\Lo - \lambda I)Ef
	\label{local}
\end{align}
for $f \in \widehat{X}$ is an immediate consequence of (\ref{eqn:Low}).

\begin{proof} Let $u = (\Lo - \lambda I)^{-1} Ef$ for $f \in \widehat{X}$. We decompose $u = ERu + (I-ER)u$. Lemma \ref{lemma:proj} implies that both summands are contained in $D$, so we can write
\begin{align*}
	Ef = ER(\Lo - \lambda I)u + (\Lo - \lambda I)(u - ERu).
\end{align*}
This demonstrates that $(\Lo - \lambda I)(u - ERu)$ vanishes in $\widehat{\Omega}^c$ (because it is in the range of $E$), and by definition it vanishes in $\widehat{\Omega}$, hence
\begin{align*}
	(\Lo - \lambda I)(u - ERu) = 0.
\end{align*}
Since $\lambda$ is in the resolvent set of $\Lo$ we find that $u - ERu = 0$, which completes the proof.
\end{proof}

We are now ready to compare the spectra of $\Lo$ and $\Lw$.

\begin{lemma} The spectra satisfy $\sigma(\Lw) \subseteq \sigma(\Lo)$, with equality if (\ref{hyp:period}) is satisfied.
\end{lemma}

\begin{proof} We proceed by studying the resolvent sets of the two operators. First suppose that $\lambda \in \rho(\Lo)$, so that $\Lo - \lambda I : D \rightarrow X$ has a bounded inverse. Using (\ref{LWequiv}), (\ref{invlocal}) and (\ref{local}), it is easily verified that $\lambda \in \rho(\Lw)$, with
\begin{align*}
	(\Lw - \lambda I)^{-1} = R (\Lo - \lambda I)^{-1} E.
\end{align*}

Now suppose (\ref{hyp:period}) is satisfied and $\lambda \in \rho(\Lw)$, so that $\Lw - \lambda I :\widehat{D} \rightarrow \widehat{X}$ has a bounded inverse. Using (\ref{eqn:Low}) and (\ref{L0equiv}), it is easily verified that $\lambda \in \rho(\Lo)$, with
\begin{align*}
	(\Lo - \lambda)^{-1} = E(\Lw - \lambda I)^{-1} R - \lambda^{-1} (I - ER ).
\end{align*}
\end{proof}


\section{The streamline decomposition}
\label{sec:geom}
Having related the essential spectrum of $\Lv$ to the spectrum of the restricted operator $\Lw$, we proceed with our geometric computation of the latter. Recalling the index set $\mathcal{J}$ defined in (\ref{eqn:index}), we have by the co-area formula (see, for instance, \cite{C06})
\begin{align}
	\int_{\widehat{\Omega}} f^2 dx dy &= \int_{-\infty}^{\infty} \left( \int_{\psi_0^{-1}(\rho)} \frac{f^2}{|\nabla \psi_0|} ds \right) d \rho \nonumber \\
	&= \int_{\mathcal{J}} \left( \int_{\Gamma(\rho)} \frac{f^2}{|\nabla \psi_0|} ds \right) d \rho
	\label{eqn:coarea1}
\end{align}
for any $f \in L^2(\widehat{\Omega})$, where $ds$ denotes the induced measure on each streamline and $\Gamma(\rho)$ is the periodic orbit corresponding to $\rho \in \mathcal{J}$.

By assumption, each $\Gamma(\rho)$ is diffeomorphic to $\mathbb{S}^1$, with period $T(\rho)$. Letting $t$ denote time under the $\psi_0$-flow, the angular coordinate $\theta := 2 \pi t / T(\rho)$ satisfies
\begin{align*}
	d \theta &= \frac{2 \pi}{T(\rho)} dt \\
	&= \frac{2 \pi}{T(\rho)} \frac{1}{|\nabla \psi_0|} ds,
\end{align*}
hence (\ref{eqn:coarea1}) can be written as
\begin{align*}
	\| f \|_{\widehat{X}}^2 = \int_{\mathcal{J}}  \left( \int_{\Gamma(\rho)} f^2 d\theta \right) d \mu(\rho)
\end{align*}
where we have defined the measure $d\mu(\rho) = (T(\rho)/ 2 \pi) d \rho$ on each component of $\mathcal{J}$. Thus in the direct integral notation of \cite{RS78} we have shown that
\begin{align*}
	\widehat{X} = \int_{\mathcal{J}}^{\oplus} L^2(\mathbb{S}^1) d\mu(\rho).
\end{align*}

We apply this geometric decomposition to the operator $\Lw$. To each $\rho \in \mathcal{J}$ we associate an unbounded operator $A(\rho)$ on $L^2(\mathbb{S}^1)$, with domain $H^1(\mathbb{S}^1)$, given by
\begin{align*}
	A(\rho) = -\frac{2\pi i}{T(\rho)} \frac{d}{d\theta}.
\end{align*}
Since each $A(\rho)$ is self-adjoint, and the period function $T: \mathcal{J} \rightarrow (0,\infty)$ is continuous, there exists a self-adjoint, unbounded operator $A$ (following the construction in Section XIII.16 of \cite{RS78}) on $L^2(\widehat{\Omega})$ given by $(A\omega)(\rho) = A(\rho) \omega(\rho)$, with domain
\begin{align*}
	D(A) = \left\{ \omega \in L^2(\widehat{\Omega}) : \omega(\rho) \in D(A(\rho)) \text{ a.e.}, \int_{\mathcal{J}} \| A(\rho) \omega(\rho) \|_{L^2(\mathbb{S}^1)}^2 d\mu < \infty \right\}.
\end{align*}

The main technical result of this section is that the operator $A$ defined above coincides with $\Lw$, up to a scalar multiple.

\begin{lemma} The operator $A$ satisfies $D(A) = \widehat{D}$ and $A = i\Lw$.
\label{lem:Asplit}
\end{lemma}

\begin{proof}  Letting $\phi_t$ denote the flow coming from the velocity field $- \nabla^{\bot}\psi_0$, we have
\begin{align*}
	\Lw f &= - \left. \frac{d}{dt}\right|_{t=0} f(\phi_t(\cdot)) \\ &= - \frac{2 \pi}{T(\rho)} \frac{\partial f}{\partial \theta}
\end{align*}
when $f$ is smooth on the closure of $\widehat{\Omega}$.

If $\omega \in D(A)$, then $\omega(\rho) \in H^1(\mathbb{S}^1)$ for almost all $\rho$, hence we can integrate by parts on each streamline to find
\begin{align*}
	\int_{\Gamma(\rho)} (A\phi)(\rho) \omega(\rho) d\theta = \int_{\Gamma(\rho)} \phi(\rho) (A\omega)(\rho) d\theta,
\end{align*}
when $\phi$ is a smooth test function on the closure of $\widehat{\Omega}$. Integrating over $\mathcal{J}$, we see that $\left<A\phi,\omega\right> = \left<\phi, A\omega \right>$, and $A\omega \in L^2(\Omega)$. Since $A\phi = i \Lw \phi$, we conclude that $\omega \in \widehat{D}$, with $\Lw \omega = -iA\omega$.

Conversely, suppose that $\omega \in \widehat{D}$. Then by definition there exists $f \in L^2(\widehat{\Omega})$ such that
\begin{align*}
	-i \int_{\mathcal{J}} \int_{\Gamma(\rho)} (A\phi)(\rho) \omega(\rho) d\theta d\mu(\rho) = 	\int_{\mathcal{J}} \int_{\Gamma(\rho)} \phi(\rho) f(\rho) d\theta d\mu(\rho)
\end{align*}
for every smooth test function $\phi$. This implies
\begin{align*}
	-i \int_{\Gamma(\rho)} [A(\rho) \phi(\rho)] \omega(\rho) d\theta = \int_{\Gamma(\rho)} \phi(\rho) f(\rho) d\theta
\end{align*}
a.e., 
hence $\omega(\rho) \in H^1(\mathbb{S}^1)$ with $-iA(\rho)\omega(\rho) = f(\rho)$. Therefore $\omega \in D(A)$, with $-iA\omega = \Lw \omega$.
\end{proof}


\section{Proof of the main theorem}
With the streamline decomposition established, we can now give the proof of Theorem \ref{thm:main}. We divide the proof into two cases, according to whether or not Hypothesis \ref{hyp:period} is satisfied.

If it is, then $\sigma_{\textrm{ess}}(\Lv; X_0) = \sigma(\Lw; \widehat{X})$. Applying Lemma \ref{lem:Asplit} above, and Theorem XIII.85 of \cite{RS78}, we have that $i \lambda \in \sigma(\Lw; \widehat{X})$ if and only if for all $\epsilon > 0$ the set
\begin{align*}
	\left\{ \rho : \sigma(A(\rho); L^2(\mathbb{S}^1)) \cap (\lambda- \epsilon, \lambda + \epsilon) \neq \emptyset \right\} \subset \mathcal{J}
\end{align*}
has positive $\mu$-measure. An elementary computation shows that
\begin{align*}
	\sigma(A(\rho); L^2(\mathbb{S}^1)) = \frac{2 \pi}{T(\rho)} \mathbb{Z}
\end{align*}
for each $\rho \in \mathcal{J}$, hence
\begin{align*}
	\left\{ \rho : \sigma(A(\rho); L^2(\mathbb{S}^1)) \cap (\lambda- \epsilon, \lambda + \epsilon) \neq \emptyset \right\}	&= \bigcup_{k \in \mathbb{Z}} T^{-1} \left( \frac{2 \pi k}{\lambda + \epsilon}, \frac{2 \pi k}{\lambda - \epsilon} \right).
\end{align*}
Therefore $\lambda \in \sigma(i \Lw; \widehat{X})$ if and only if for every $\epsilon > 0$ there exists $k \in \mathbb{Z}$ with
\begin{align}
	\mu \left[ T^{-1} \left( \frac{2 \pi k}{\lambda + \epsilon}, \frac{2 \pi k}{\lambda - \epsilon} \right) \right] > 0.
	\label{spec:nsc1}
\end{align}

We now prove that this condition is satisfied precisely when there exist $\{\rho_n\}$ and $\{k_n\}$ such that 
\begin{align}
	\lambda = \lim_{n \rightarrow \infty} \frac{2\pi k_n}{T(\rho_n)}.
	\label{spec:nsc2}
\end{align}
If (\ref{spec:nsc1}) is satisfied, then for each $n$ there exists $k_n \in \mathbb{Z}$ such that
\begin{align*}
	\left\{ \rho : \frac{2 \pi k_n}{\lambda + 1/n} < T(\rho) <  \frac{2 \pi k_n}{\lambda - 1/n} \right\}
\end{align*}
has positive $\mu$-measure, and hence contains an element, say $\rho_n$. Therefore (\ref{spec:nsc2}) holds.

Conversely, suppose (\ref{spec:nsc2}) is satisfied. By compactness there exists a subsequence $\{\rho_n\}$ converging to an element in $\overline{\mathcal{J}}$. Let $\epsilon > 0$. Then for $n$ sufficiently large we have
\begin{align*}
	\frac{2\pi k_n}{\lambda+\epsilon} < T(\rho_n) < \frac{2\pi k_n}{\lambda-\epsilon}.
\end{align*}
It follows from the continuity of $T$ that 
\begin{align*}
	T^{-1} \left( \frac{2 \pi k_n}{\lambda + \epsilon}, \frac{2 \pi k_n}{\lambda - \epsilon} \right)
\end{align*}
contains an open interval around $\rho_n$ and hence has positive $\mu$-measure.

We have thus established (\ref{spec}) in the case that (\ref{hyp:period}) is satisfied. If the flow has arbitrarily long orbits, this immediately yields $\sigma_{\textrm{ess}}(\Lv; X_0)= i \mathbb{R}$.

It remains to deal with the case that Hypothesis \ref{hyp:period} is not satisfied. In this case there necessarily exists an aperiodic trajectory in $\Omega$. Thus for any $N \in \mathbb{N}$ there is a bounded orbit $\Gamma_N$ of length $\geq N$, with the distance between $\Gamma_N$ and $\partial \Omega$ bounded away from zero. Then we can use the approximate eigenfunction construction of \cite{LS05} (which is supported in a tubular neighborhood of the streamline $\Gamma_N$), to show that $\sigma_{\textrm{ess}}(\Lv; X_0)= i \mathbb{R}$. This completes the proof of Theorem \ref{thm:main}.


\bibliographystyle{plain}
\bibliography{euler}

\end{document}